\documentclass[11pt]{amsart}

\usepackage{amsmath}
\usepackage{url}
\usepackage{amsthm}
\usepackage{amsfonts}
\usepackage{hyperref}
\usepackage{accents}
\usepackage{xcolor}

\newtheorem{theorem}{Theorem}[section]
\newtheorem{lemma}[theorem]{Lemma}
\newtheorem{proposition}[theorem]{Proposition}
\newtheorem{corollary}[theorem]{Corollary}

\newtheorem{prop}{Proposition}[section]
\newtheorem{remark}[prop]{Remark}
\newtheorem{definition}[prop]{Definition}

\makeatletter \@addtoreset{equation}{section} \makeatother

\def\<{\langle}
\def\>{\rangle}

\begin{document}

\title[]{On the Morse Index with Constraints I: An Abstract Formulation}

\author{Hung Tran}
\address{Department of Mathematics and Statistics,
	Texas Tech University,
	Lubbock, TX 79409}

\author{Detang Zhou}
\address{Instituto de Matem\'atica, UFF, Rua Professor Marcos Waldemar de Freitas Reis, Bloco H - Campus do Gragoat\'a, S\~ao Domingos, 24.210-201, Niter\'oi, RJ - BRAZIL}
\thanks{The first author was partially supported by a Simons Foundation Collaboration Grant}
\thanks{The second author was partially  supported by Faperj and CNPq of Brazil.}
\date{}

\maketitle
\begin{abstract} In this sequence, we first prove an abstract Morse index theorem in a Hilbert space modeling a variational problem with constraints. Then, our abstract formulation is applied to study several optimization setups including closed CMC hypersurfaces, capillary surfaces in a ball, and critical points of type-II partitioning. In this paper, we study the index and nullity of a symmetric bounded bilinear form in a Hilbert space. The main results determine precisely how these notions change when restricting to a subspace of a finite codimension.    
	
\end{abstract}

\section{Introduction}

Variational problems in geometry, a main theme in mathematics, consider certain functionals on geometrical subjects, such as a hypersurface or a manifold. From the viewpoint of Morse theory, it is essential to study  the second variation at critical points of a functional which normally involves a symmetric bilinear form in a suitable function space.  One is interested in determining quantitatively how negative the bilinear form could be, leading to the notion of the Morse index. The pioneered classical Morse index investigation has been done by H. Edwards \cite{edwards64}, S. Smale \cite{smale65}, J. Simons \cite{Simons68}, K. Uhlenbeck \cite{uhlenberg73} and others. Recently, it plays a crucial role in the resolution of the Willmore conjecture by F. Marques and A. Neves \cite{MN_minmax14}. \\ 

The index of a bilinear form in a vector space is defined to be the maximum dimension of a negative definite subspace with respect to the form.  Intuitively, it gives the number of distinct deformations which decrease the functional to the second order. The choice of the defined space depends on the geometric problems and given constraints. For example, ancient mathematicians like Zenodorus and Princess Dido considered isoperimetric problems, finding the largest possible shape with a given perimeter. Thus, using the variational point of view, one considers only variations fixing the boundary measurement. A generalized version is the double bubble conjecture which was only resolved recently by M. Hutchings, F. Morgan, M. Ritore, and A. Ros \cite{HMRR02}. Another example is the partitioning problem of a convex body by least-area hypersurfaces under a type I or type II constraint using the terminology from \cite{BS79}. Type I requires the partitions to have prescribed volume while type II preserves the wetting boundary area. The former is particularly popular in literature, normally associated with constant-mean-curvature (CMC) hypersurfaces. \\

Consequently, one is motivated to study the index of the second variation restricted to deformations satisfying certain constraints. In case of volume-preserving consideration, that is the so-called weak Morse index. The stability case, when the index is zero, has received plenty of interests; see, for example, \cite{HC18, WX19, BCJ88}. Since the constraint is linear, it is easy to see that the difference between the weak Morse index and the general Morse index is zero or one. When the Morse index is zero, obviously the weak Morse index is also zero. There are some special cases when they are equal to each other (and non zero) such as \cite{BB00} (for a non-compact, infinite volume, CMC immersion into a hyperbolic space) and \cite{LR89} (for catenoids and Enneper surfaces with index one). However, one suspects that they are not equal to each other in general \cite{koise01}, \cite{vogel87}, \cite{souam19}. In this sequence of papers, we will identify criteria determining the relation between these index notions. \\
  
Precisely, in this first paper, we study the index theory of a symmetric continuous bilinear form in an abstract Hilbert space, investigating how the index changes when restricting to a subspace of finite co-dimension. One can consider our approach as a local version of the theory developed in \cite{hestenes51}. Instead of an assumption on the bilinear form, we look at the characterization of the complementary subspace.  Then, in a subsequent paper, we apply our abstract formulation to study CMC and minimal hypersurfaces, with and without boundaries. In the subsequent paper, we'll apply our abstract formulation to study critical points of several variational problems in geometry. \\

We first recall the Morse index and nullity of a bilinear form which are applicable for any vector space. 
\begin{definition}
	The Morse index of a bilinear form  $S(\cdot, \cdot)$ in a vector space $V$, $\text{MI}(S, V)$, is the maximal dimension of a subspace of $V$ on which $S(\cdot,\cdot)$ is negative definite. The nullity, $n(S, V)$, is the dimension of the radical of the bilinear form; that is, the set of all $u\in V$ such that $S(u,v)=0$ for all $v\in V$. 
\end{definition}

Motivated by variational problems with constraints, we give the following definition. Let  $\phi_i$ be a linear functional on $V$ with kernel $\text{Ker}(\phi_i)$.  
\begin{definition}
	The Morse index of the bilinear form  $S(\cdot, \cdot)$ with respect to $\phi_1,....\phi_n$, $\text{MI}^{\phi_1,...\phi_n}(S)$, is the index of $S(\cdot,\cdot)$ in $\cap_{i=1}^n\text{Ker}(\phi_i)$, i.e. 
	\[\text{MI}(S, \cap_{i=1}^n\text{Ker}(\phi_i)).\] 
	The nullity of the bilinear form  $S(\cdot, \cdot)$ with respect to $\phi_1,....\phi_n$, $n^{\phi_1,...\phi_n}(S)$, is the nullity $S(\cdot,\cdot)$ in $\cap_{i=1}^n\text{Ker}(\phi_i)$, i.e 
	\[n(S, \cap_{i=1}^n\text{Ker}(\phi_i)).\] 
\end{definition}

To study such an index with constraints, one way is to relate these $\phi_i$'s to elements of the vector space $V$. The Hilbert space formulation exactly provides that bridge. Furthermore, it is easy to see that the index does not change going from a vector space $V$ to a Hilbert space $H$, see Lemma \ref{vtoH}, if $V$ is dense in $H$.

\begin{remark}
	When the context is clear, to simplify notation, we'll drop the space factor. That is, we write $\text{MI}(S)$ instead of $\text{MI}(S, \cdot)$. 
\end{remark}

Let $H$ be a separable Hilbert space with an inner product $(\cdot, \cdot)$. $S(\cdot, \cdot)$ is a symmetric continuous bilinear form and $\phi_i$ is a continuous linear functional on $H$.
The inner product on $H$ induces a linear map $\mathcal{S}$ from $H$ to its continuous dual $H^\ast$ such that, for all $v\in H$, 
\[S(u,v)=(\mathcal{S}u)(v).\] 
Via the Riesz representation theorem, $H^\ast$ can be equipped with an inner product so that it is isometric to $H$. 

\begin{definition} Let $\text{ran}(\mathcal{S})$ be the range of $\mathcal{S}$, $\overline{\text{ran}(\mathcal{S})}$ its closure by the induced norm. $\overline{\text{ran}(\mathcal{S})}-\text{ran}(\mathcal{S})$ is called the set of pure limit points. 
\end{definition}
\begin{remark}
	When the associated self-adjoint operator has closed range then the set of pure limit points is empty. 
\end{remark}
One notes that a subspace of finite co-dimension of a Hilbert space is closed and, hence, a Hilbert space. 
We are now ready to state the main theorem. 	
\begin{theorem}
	\label{abstractMorse}
	Let $H$ be a separable Hilbert space and $S(\cdot, \cdot)$ is a continuous symmetric bilinear form. Then for any non-trivial continuous linear functional $\phi$ such that $\phi$ is not a pure limit point, we have 
	\[ \text{MI}^\phi(S)=\begin{cases}
	\text{MI}(S)-1 \text{   if there is $u\in H$ such that $\mathcal{S}u=\phi$ and $\phi(u)\leq 0$}\\
	\text{MI}(S) \text{  otherwise.}
	\end{cases}
	\]
	\end{theorem}
\begin{remark}
	Our convention is that $\infty-1=\infty$. Thus, for the case of an infinite index, the result is vacuously true. 
\end{remark}
\begin{remark}
	The assumption that $\phi$ is not a pure limit point turns out to be quite natural in PDE settings, essentially equivalent to the existence of a Fredholm alternative. Nevertheless, from a purely functional analysis point of view, it is interesting to consider $\phi\in \overline{\text{ran}(\mathcal{S})}-\text{ran}(\mathcal{S})$.  
\end{remark}

A notion of stability is generally associated with the case $\text{MI}(S)=0$. The following immediate consequence is an indicator of in-stability. 
\begin{corollary}
	If there is a function $u$ such that $\mathcal{S}u=\phi$, $\vec{0}\neq \phi$, and $\phi(u)\leq 0$ then $\text{MI}(S)>0$. 
\end{corollary}

It is also possible to generalize the result from a single $\phi$ to several. 
\begin{theorem}
	\label{severalfunc}
	Let $H$ be a separable Hilbert space and $S(\cdot, \cdot)$ is a continuous symmetric bilinear form. Suppose that, for $i=1,...n$,
	\[\mathcal{S}(u_i)=\phi_i,\]
	and $\{\phi_i\}_{i=1}^n$ are linearly independent. Then, \[\text{MI}^{\phi_1,...,\phi_n}(S)=\text{MI}(S)-c,\]
	where $c$ is the number of non-positive eigenvalues of the symmetric matrix $S(u_i,u_j)$. In particular, 
	\[ \text{MI}(S)\geq c.\]
\end{theorem}
Similarly, the following statements address the nullity with a constraint. They will not be used in the subsequent paper but are presented for independent interests.  
\begin{theorem} 
	\label{abstractnullity}
	Let $H$ be a separable Hilbert space, $S(\cdot, \cdot)$ be a continuous symmetric bilinear form, and  $\phi$ is a  nonzero continuous linear functional. Then, we have the followings:
	\begin{enumerate}
		\item If there is $u\in H$ such that $\mathcal{S}u=\phi$, $\phi(u)=0$, then $n^\phi(S)=n(S)+1$; 
		\item If $\phi\notin \overline{\text{ran}(\mathcal{S})}$, $n^\phi(S)=n(S)-1$;\\
		\item Otherwise, $n^\phi(S)=n(S)$.
	\end{enumerate}
\end{theorem}

\begin{theorem}
	\label{severalfuncnullity}
	Let $H$ be a separable Hilbert space and $S(\cdot, \cdot)$ is a continuous symmetric bilinear form. Suppose that, for $i=1,...n$,
	\[\mathcal{S}(u_i)=\phi_i,\]
	and $\{\phi_i\}_{i=1}^n$ are linearly independent. Then, \[n^{\phi_1,...,\phi_n}(S)=n(S)+c,\]
	where $c$ is the dimension of the null space of the symmetric matrix $S(u_i,u_j)$. 
\end{theorem}

\section{Preliminaries}
\label{prelim}
First, we record our notations and conventions. 
\begin{itemize}
	\item Let $S(\cdot,\cdot)$ be a symmetric, continuous, bilinear form in a Hilbert space $H$ with an inner product $(\cdot,\cdot)$. $\mathcal{S}$ is the induced linear map from $H$ to its continuous dual $H^\ast$. $\mathbf{S}$ is the associated self-adjoint operator from $H$ to itself. $\mathfrak{S}$ is the associated quadratic form. 
	\item $\vec{0}$ is the zero vector in a Hilbert space.
	\item $\phi$ is a continuous linear functional then $\bar{\phi}\in H$ is the correspondence due to the Riesz representation theorem.

\end{itemize}

 Let $S(\cdot, \cdot)$ be a symmetric continuous bilinear form on a separable Hilbert space $H$ with an inner product $(\cdot, \cdot)$. Equivalently, $S(\cdot, \cdot)$ is bounded. That is, for any $u, v\in H$ there is a universal constant $c$ such that 
 \[S(v,u)=S(u,v)\leq c ||u||||v||.\]
 
 Due to the Riesz representation theorem, we can identify $H$ with its continuous dual of linear continuous functionals via the isomorphism
\[\pi: u\mapsto u^\ast: u^\ast(v)=(u,v).\]
The inner product on $H^\ast$ is given by
\[(u^\ast, v^\ast)=(u,v).\]
Obviously, $||u^\ast||_{H^\ast}= ||u||_{H}$. \\

It is well-known that $S$ is totally determined by its quadratic form 
\[\mathfrak{S}(v):=S(v,v).\] 
Also, it induces a linear map $\mathcal{S}$ from $H$ to its continuous dual $H^\ast$ such that, for all $v\in H$, 
\[S(u,v)=\mathcal{S}u(v).\] 
Similarly, due to the Riesz representation theorem, one can also think of an associated self-adjoint operator $\mathbf{S}: H \mapsto H$ such that, for all $v\in H$,
\[S(u,v)= (\mathbf{S}u, v).\]
We have, 
\[||\mathbf{S}u||^2=(\mathbf{S}u, \mathbf{S}u)=S(u, \mathbf{S}u)\leq c||u||||\mathbf{S}u||.\]
Therefore, $||\mathbf{S}u||\leq c||u||$ and $\mathbf{S}$ is a bounded self-adjoint operator. For $\mathbf{S}$, let $\text{ran}(\mathbf{S})$ denote its range and $\text{Ker}(\mathbf{S})$ its kernel. It is immediate that 
\[\text{Ker}(\mathbf{S})=\text{Ker}(\mathcal{S}).\]

Furthermore, one can define a projection operator with respect to such a bilinear form. In particular, 
\[ Pro^S_u(v)=\frac{S(u,v)}{S(u,u)}u\]
is the projection of $u$ onto $v$ with respect to $S$ for any $S(u, u)\neq 0$. Since $H$ contains a countable basis, by adjusting the Gram-Schmidt orthogonalization process appropriately, one can diagonalize $H$ with respect to $S$. The following definitions illustrate some geometry induced by $S(\cdot, \cdot)$.

\begin{definition} We have the followings:
	\begin{itemize}
		\item A vector $u\in H$ is $S$-perpendicular to $v\in H$ if $S(u,v)=0$. 
		\item A vector $v \in H$ is called isotropic if $S(v,v)=0$. Otherwise, it is called non-isotropic. 
		\item A subspace $W\subset H$ is called isotropic if $S \mid_{W\times W}=0$. 
		\item For any subspace $W\subset H$, its $S$-perpendicular subspace is defined as 
		\[ W^\perp=\{v\in H| S(v, v_1)=0 ~~\forall v_1\in W\}.\]
	\end{itemize}
\end{definition}

Thus, a subspace $W$ is isotropic if and only if $W\subset W^\perp$. Moreover, if $W$ contains no isotropic vector then 
\[H=W\oplus W^\perp,\]
where the direct product is with respect to $S$-perpendicularity. 

 Also, since $\mathbf{S}$ is a bounded self-adjoint operator, the closure of its range is exactly the $S$-orthogonal complement of its kernel. That is,
\begin{align*}
\overline{\text{ran}(\mathbf{S})}&=(\text{Ker}(\mathbf{S}))^\perp,\\
H &=\overline{\text{ran}(\mathbf{S}}) \oplus \text{Ker}(\mathbf{S}).
\end{align*}
The direct sum here is with respect to either the inner product or the bilinear form $S(\cdot, \cdot)$. Correspondingly, in the context of the isomorphism $\pi:H\mapsto H^\ast$
\[H^\ast = \pi(\text{Ker}(\mathbf{S}))\oplus \overline{\text{ran}(\mathcal{S})}.\]
\begin{definition}
	$ \overline{\text{ran}(\mathcal{S})}-\text{ran}(\mathcal{S})$ is called the set of pure limit points. 
\end{definition}
It is well-known that, for a normal operator with $0$ in its spectrum, its range is closed if and only if $0$ is not a limit point of the spectrum \cite[Proposition XI.4.5]{conwaybook90}. For a general bounded operator, a generalized version is given by \cite{KN00} considering the spectrum of the product of the operator and its adjoint. \\

The following well-known result gives a fundamental decomposition of $H$ with respect to $S(\cdot, \cdot)$.
\begin{theorem}\cite[Theorem 7.1]{hestenes51}
	\label{funddecom}
	Given a symmetric bilinear form $S(\cdot, \cdot)$, $H$ can be decomposed uniquely as the direct sum
	\[H=H_{-}\oplus H_0\oplus H_+\]
	satisfying the following properties
	\begin{itemize}
		\item $H_{-}$, $H_0$, and $H_+$ are mutually perpendicular and $S$-perpendicular.
		\item $S(\cdot, \cdot)$ is negative definite on $H_{-}$, zero on $H_0$, and positive definite on $H_+$.
	\end{itemize}
\end{theorem} 
That is, \[H_0=\text{Ker}(\mathbf{S}),\]
and \[\text{MI}(S)=\text{dim}(H_) \text{  and} ~~~ n(S)=\text{dim}(H_0).\] 

A priori, $H_{-}$ and $H_0$ can be subspaces of infinite dimensions. However, in applications, we're only interested in the case that $\text{dim}(H_{-})$ and $\text{dim}(H_{0})<\infty$. That is guaranteed if the associated quadratic form $\mathfrak{S}$ is Legendre \cite[Theorem 11.2]{hestenes51} or the associated self-adjoint operator $\mathbf{S}$ is compact. Another case, which arrives frequently in PDE theory, is that the bilinear form is related to a compact operator constructed by the inverse of some isomorphism between a Hilbert space and its continuous dual via the Riesz representation theorem. \\

As a consequence, we assume that finiteness condition for the rest of the paper. In that case,  for any maximal space $W$ on which $S(\cdot, \cdot)$ is negative definite, the $S$-projection from $W$ to $H_{-}$ is an isomorphism. Thus, $\text{dim}(W)=\text{dim}(H_{-})$ and the following definition is unambiguous. 

\begin{definition}
	\label{strongMI}
	The index of bilinear form $S(\cdot, \cdot)$ on a Hilbert space $H$, denoted by $\text{MI}(S)$, is the dimension of a any maximal subspace of $H$ on which the second variation is negative definite. 
\end{definition}
Consequently, for any vector $u$, which is $S$-perpendicular to a maximal subspace, $S(u,u)\geq 0$. 

\begin{remark}
	In a similar way, one might consider the dimension of a maximal space on which $S(\cdot, \cdot)$ is totally vanishing. It can be shown that the number is equal to the nullity plus the index.  
\end{remark}

As mentioned in the introduction, the index of bilinear form can be defined in any vector space. In our applications, we move back and forth between a vector $V$ which is dense in a Hilbert space $H$. The justification is provided by the following lemma, which is well-known.
\begin{lemma}
	\label{vtoH}
	Let $V\subset H$ is a dense vector space inside a Hilbert space. Let $S(\cdot,\cdot)$ is a continuous symmetric bilinear form on $H$. Then, the index of $S(\cdot, \cdot)$ on $H$ is equal to that on $V$.
\end{lemma}
A proof is provided as we couldn't find a reference.
\begin{proof}
	Let $m_H, m_V$ denotes the indices of $S(\cdot,\cdot)$ on $H$ and $V$, respectively. By the definition, \[m_H\geq m_V.\]
	To prove the reverse inequality, we proceed by contradiction. Suppose that $m_H>m_V$. Let $W\subset V$ be a maximal space of $m_V$. Since, $m_H>m_V$, $W$ is not a maximal space of $m_H$. That is, there is $u\in H$ such that $u$ is perpendicular to $W$ and $S(u,u)<0$. 
	
	Since $V$ is dense in $H$, there is a sequence $u_n\in V$ such that $||u_n-u||_H\rightarrow 0$. Furthermore, let $v_n$ be the projection of $u_n$ on $W$ then
	\[||u_n-u||^2=||v_n||^2+||(u_n-v_n)-u||^2.\]
	Thus, for $u'_n=u_n-v_n\in V$, $u'_n$ is perpendicular to $W$, $||u_n'-u||_H\rightarrow 0$. We have
	\[S(u_n', u_n')=S(u,u)+S(u_n'-u,u+ u_n').\]
	Since $S(\cdot,\cdot)$ is continuous and $||u_n'-u||_H\rightarrow 0$, $S(u_n'-u,u+ u_n')\rightarrow 0$. Hence, for sufficiently large $n$, 
	\[S(u_n', u_n')<0,\]
	which is a contradiction to the maximality of $W\subset V$. The proof is finished.  
	  
	\end{proof} 
Moreover, for variational problems with constraints, admissible deformations generally must satisfy some linear condition. For example, the partitioning of a convex body with prescribed volume is equivalent to considering deformations with average zero. Thus, in the context of a Hilbert space, one is motivated to consider the following generalization. It is noted that a subspace of finite codimension in a Hilbert space is a Hilbert space itself.  
\begin{definition}
	The index of the bilinear form  $S(\cdot, \cdot)$ with respect to continuous linear functionals $\phi_1,....\phi_m$, $\text{MI}^{\phi_1,...\phi_m}(S)$, is the index of $S(\cdot,\cdot)$ in $\cap_{i=1}^m\text{Ker}(\phi_i)$. The nullity of the bilinear form  $S(\cdot, \cdot)$ with respect to continuous linear functionals $\phi_1,....\phi_m$, $n^{\phi_1,...\phi_m}(S)$, is the nullity $S(\cdot,\cdot)$ in $\cap_{i=1}^m\text{Ker}(\phi_i)$. 
\end{definition}
To shed light on these quantities, we first observe that, for a continuous linear functional $\phi$, $\text{Ker}(\phi)$ is a subspace of co-dimension at most one. The restriction of $S(\cdot,\cdot)$ and $\mathcal{S}$ on $\text{Ker}(\phi)$ is denoted by the same notation as the following equation holds true for every $u,v\in \text{Ker}(\phi)$
\[S(u,v)=\mathcal{S}u(v).\]
\begin{remark}
That is an advantage of having the abstract continuous dual $H^\ast$. In comparison, the restriction of $\mathbf{S}$ on $\text{Ker}(\phi)$ is the composition of $\mathbf{S}$ and a projection into $\text{Ker}(\phi)$.
\end{remark}
The following is immediate. 
\begin{lemma}\label{atmost1}
	$\text{MI}^\phi(S)$ is either $\text{MI}(S)$ or $\text{MI}(S)-1$. Similarly, $n^{\phi}(S)\geq n(S)-1$.  
\end{lemma}
\begin{proof}
	Obviously, by definition,
	\[ \text{MI}^\phi(S)\leq \text{MI}(S).\]
	On the other hand, take $W$ be a maximal subspace on which $S(\cdot, \cdot)$ is negative definite. Then $\text{Ker}(\phi\mid_W)$ has co-dimension at most one. Obviously, $S(\cdot, \cdot)$ is positive definite on $\text{Ker}(\phi\mid_W)$ which is certainly a subspace of $\text{Ker} (\phi)$. Therefore  
	\[ \text{MI}^\phi(Q)\geq \text{MI}(Q)-1.\]
	The result then follows. The statement for the nullity follows from a similar argument. 
\end{proof}
\begin{remark}
	It is interesting that the nullity might increase when restricting to a smaller subspace. See Theorem \ref{abstractnullity}.
\end{remark}
To determine the relation between $\text{MI}^{\phi}(S)$ with $\text{MI}(S)$, it is essential to consider the effect of $\phi$ on maximal subspaces. It leads to the following.  
\begin{definition}
	A continuous linear function $\phi$ is called $S$-critical if for any maximal subspace $W$ on which $S(\cdot, \cdot)$ is negative definite then $\phi(W)=\mathbb{R}$.  
\end{definition}

\section{Index of a BiLinear Form in a Hilbert Space}
\label{abstract}
In this section, we prove an abstract theorem for the Morse index on a Hilbert subspace and related results. Essentially, we determine the relation between $\text{MI}(S)$ and $\text{MI}^\phi(S)$ for a non-trivial continuous linear functional $\phi$.  Precisely, the main statement is as follows. 
\begin{theorem}
	\label{Scriticality}
	Let $H$ be a separable Hilbert space and $S(\cdot, \cdot)$ be a continuous symmetric bilinear form. Assume that  $\phi$ is a  nonzero continuous linear functional which is not a pure limit point. Then the following assertions are equivalent:
\begin{enumerate}	
	\item    $\text{MI}^\phi(S)=\text{MI}(S)-1$; 
	
	\item  There exists  a $u\in H$ such that $\mathcal{S}u=\phi$ and $\phi(u)\leq 0$.
	
	\item  $\phi$ is $S$-critical i.e. $\phi(W)=\mathbb{R}$ for any maximal subspace $W$ on which $S(\cdot, \cdot)$ is negative definite. 
\end{enumerate}
\end{theorem}

\begin{remark}
	The assumption on the non-triviality of $\phi$ is necessary. If $\phi=\vec{0}$ then, obviously, $\text{Ker}(\phi)=H$ and $\text{MI}^\phi(S)=\text{MI}(S)$. However, for $u\in \text{Ker}(\mathcal{S})$, we have $\mathcal{S}(u)=\phi$ and $\phi(u)=0$. 
\end{remark}

\begin{proof}[Proof of Theorem \ref{Scriticality}.]
First we prove that (1) is equivalent to (3).  Suppose that $\phi$ is $S$-critical.  If $\text{MI}^\phi(S)=\text{MI}(S)$ then there is a maximal subspace $W$ of that dimension in $\text{Ker}\phi$ on which $S(\cdot,\cdot)$ is negative definite. That is, $\phi(W)=0$, a contradiction with the definition of $S$-criticality. Therefore, by Lemma \ref{atmost1}, $\text{MI}^\phi(S)=\text{MI}(S)-1$. 
	
	Conversely, suppose that $\text{MI}^\phi(S)=\text{MI}(S)-1$ and $\phi$ is not $S$-critical. Then, there is a maximal subspace $W$ on which $S(\cdot, \cdot)$ is negative definite and $\phi(W)=0$. But that means $W\subset \text{Ker}(\phi)$ and $\text{MI}^\phi(S)=\text{MI}(S)$, again a contradiction.    

	Now we start to prove (1)  is equivalent to (2).  Via the Riesz representation theorem, each $\phi$ uniquely corresponds to $\bar{\phi}\in H$ such that, for all $v\in H$,
\[\phi(v)=(\bar{\phi}, v).\]

  There are two cases: (i) $\bar\phi\notin \textrm{ran}(\mathbf{S})$ and (ii) $\bar\phi\in \textrm{ran}(\mathbf{S})$.

 {\bf Case (i).}  $\bar\phi\notin \textrm{ran}(\mathbf{S})$. We will show that $\text{MI}^\phi(S)=\text{MI}(S).$
 
 By the decomposition discussed in Section \ref{prelim}, $\bar{\phi}=u+s$ with $u\in \text{Ker}(\mathfrak{S})$ and $s\in \overline{\text{ran}(\mathfrak{S})}$. Since $\phi$ is not a pure limit point, neither is $\bar\phi$ and $\vec{0}\neq u$. Then we have 
\[\phi(u)=(\bar{\phi},u)=(s+u, u)=(u,u)>0.\]

If $\text{MI}(S)=0$, the result follows vacuously. Otherwise, let $W$ be a maximal space on which $S(\cdot, \cdot)$ is negative definite. That is,
\[ \text{dim}(W)=\text{MI}(S)>0.\]
Since $u\neq\vec{0}$ and 
\[ S(u, v)=(\mathcal{S}u)v=\vec{0}v=0,\]
we have $u\notin W$. Then let $W_1=\text{span}(W, u)$ and $W_2=\text{Ker}(\phi\mid_{W_1})$. It is clear that 
\[ \text{dim}(W_1)=\text{MI}(S)+1.\]
Since $\phi(u)\neq 0$, the map $\phi: W_1\mapsto \mathbb{R}$ is onto, then
\[ \text{dim}(W_2)=\text{dim}(W_1)-1=\text{MI}(S).\]
Let $\vec{0}\neq v\in W_2$. Then 
\[v=w+cu,\]
for $\vec{0}\neq w\in W$ and some constant $c$. Thus,
\[S(v,v)=S(w,w)+2cS(u,w)+c^2S(u,u).\]
Since $u\in \text{Ker}(\mathfrak{S})$, $S(u,w)=S(u,u)=0$. Thus, $S(v,v)=S(w,w)<0$ and $S(\cdot, \cdot)$ is negative definite on $W_2$. Consequently, $\text{MI}^\phi(S)=\text{MI}(S)$.   

{\bf Case (ii).} $\phi\in \text{ran}(\mathcal{S})$. Then there is $u\in H$ such that
\begin{align*}
\mathcal{S}u &={\phi},\\
\mathbf{S}u &=\bar{\phi},\\
S(u,u) &=(\mathbf{S}u,u)=(\bar{\phi}, u)=\phi(u).
\end{align*}
In Propositions \ref{phiupos} we show that if $\phi(u)>0$ then $\text{MI}^\phi(S)=\text{MI}(S)$ and in Propositions  \ref{phiuneg} and \ref{phiuzero} we show that if $\phi(u)<0$ or $\phi(u)=0$ then $\text{MI}^\phi(S)=\text{MI}(S)-1$.
	Therefore the equivalence between (1) and (2) follows from Propositions  \ref{phiupos}, \ref{phiuneg}, and \ref{phiuzero} below. 
\end{proof}
\begin{proposition}
	\label{phiupos}
	If $\phi(u)>0$, then $\text{MI}^\phi(S)=\text{MI}(S)$. \end{proposition}
\begin{proof} 
	Let $W$ be a maximal space on which $S(\cdot, \cdot)$ is negative definite. That is,
	\[ \text{dim}(W)=\text{MI}(S).\]
	We have
	\[S(u,u)=(\mathcal{S}u)u=\phi(u)>0.\]
	Therefore $u\notin W$. Let $W_1=\text{span}(W,u)$ and $W_2=\text{Ker}(\phi\mid_{W_1})$. It is readily checked that 
	\[ \text{dim}(W_2)=\text{dim}(W_1)-1=\text{MI}(S).\]
	Furthermore, $v\in W_2$ if and only if $v=ku+w$ for $w\in W$ and $\phi(ku+w)=0$. We calculate
	\begin{align*} 
	S(v,v) &=k^2S(u,u)+S(w,w)+2kS(u,w)\\
	&= k^2 \phi(u)+S(w,w)+2k\phi(w)\\
	&=-k^2\phi(u)+S(w,w)<0. \\
	\end{align*}
	The result then follows.
\end{proof}

So it remains to prove Propositions \ref{phiuneg} and \ref{phiuzero}. For preparation, we observe several simple results. It is noted that $u$ is generally not unique as one can replace it by $u+v$ for any $v\in \text{Ker}(\mathcal{S})$. However, $\phi(u)$ is unique as 
\[\phi(u+v)=S(u,u+v)=S(u,u).\] 
Also, we have, for $v\in \text{Ker}(\phi)$
\begin{align*}
S(u,v) &= (\mathcal{S}u)(v)=\phi(v)=0.
\end{align*}
Hence $u$ is $S$-perpendicular to $\text{Ker}(\phi)$.

\begin{lemma}
	\label{containingu}
	For any $u\in H$ such that $S(u,u)<0$, there is a maximal space $W$, $u\in W$, on which $S(\cdot,\cdot)$ is negative definite.
\end{lemma}
\begin{proof}
	Since $H_{-}$ has no isotropic vector, it is possible to construct $u_0$, the $S$-projection of $u$ on $H_{-}$. Since $S(\cdot,\cdot)$ is negative definite on $H_{-}$, we see that $H_{-}$ decomposes into 
	\[H_{-}=\text{span}(u_0)\oplus W_1,\]
	for $u_0$ $S$-perpendicular to $W_1$. As a consequence, $u$ is $S$-perpendicular to $W_1$ and $W:=\text{Span}(u,W_1)$ is a maximal space.  
\end{proof}

\begin{lemma}
	\label{maximalnotinkernel}
	For any nontrivial continuous linear functional $\phi$, if $\text{MI}(S)>0$ then there is a maximal space $W$ on which $S(\cdot,\cdot)$ is negative definite and $W\not\subset \text{Ker}\phi$.
\end{lemma}
\begin{proof}
	Due to Lemma \ref{containingu}, it suffices to find some $u\notin  \text{Ker}\phi$ such that $S(u,u)<0$. Recall, for all $v\in H$
	\[\phi(v)=(v, \bar{\phi}).\]
	Since $\phi$ is nontrivial, $\bar{\phi} \neq \vec{0}$. 
	Next, for $\text{MI}(S)>0$, there is $x$ such that
	\begin{align*}
	S(x,x) &<0.
		\end{align*}
	if $\phi(x)\neq 0$ then we are done. Otherwise, $\phi(x)=0$, one considers the quadratic function  
	\begin{align*}
	f(k) &=S(kx+\bar{\phi}, kx+\bar{\phi})\\
 &	=k^2 S(x,x)+2k S(x,\bar{\phi})+S(\bar{\phi},\bar{\phi})\rightarrow -\infty \text{ as } k \rightarrow \infty. 
	\end{align*}   
	Choose $u=kx+\bar\phi$, $\phi(u)=(\bar\phi, kx+\bar\phi)=(\bar\phi,\bar\phi)>0$, for sufficiently large $k$ and the proof is finished.
\end{proof}

\begin{proposition}
	\label{phiuneg}
	If $\phi(u)<0$ then $\text{MI}^\phi(S)=\text{MI}(S)-1$.
\end{proposition}

\begin{proof}
	Since
	\[S(u,u)=(\mathcal{S}u)u=\phi(u)<0,\]
	by Lemma \ref{containingu}, there is a maximal space $W$ containing $u$ such that $S(\cdot,\cdot)$ is negative definite on $W$. Let $W_1=\text{Ker}(\phi\mid_{W})$. By the discussion above $u$ is $S$-perpendicular to $W_1$. Thus,  
	\begin{align*}
 W &=\text{span}(u)\oplus W_1,\\
 \text{dim}(W_1) &=\text{dim}(W)-1\\
 &=\text{MI}(S)-1.
	\end{align*}
	Consequently, 
	\[H=\text{span}(u)\oplus W_1\oplus W^\perp.\]
	The direct product is with respect to $S$-perpendicularity. 
	
	\textbf{Claim:} $W_1$ is a maximal subspace inside $\text{Ker}(\phi)$ such that $S(\cdot,\cdot)$ is negative definite on $W_1$. 
	
	\textbf{Proof of the claim.} The claim is proved by contradiction. Suppose that the claim is false. Then there exist an  element $f\in \text{Ker}(\phi)$ which is  $S$-perpendicular to $W_1$, and $S(f,f)<0$. As $f$ is $S$-perpendicular to $W_1$, by the decomposition above, we have
	\[ f=ku+f_1,\] 
	for $f_1\in W^\perp$ and some constant $k$. Then we calculate
	\begin{align*}
	0 &=S(u,f_1) \\
	&=(\mathcal{S}u) (f_1)\\
	&=\phi(f_1).
	\end{align*}
	Thus, $f_1\in \text{Ker}(\phi)$. Since $f=ku+f_1\in \text{Ker}(\phi)$, $ku \in \text{Ker}(\phi)$. Recall that $\phi(u)<0$, we conclude $k=0$. Then, due to $f_1\in W^\perp$,
	\[S(f,f)=S(f_1,f_1)\geq 0.\]
	That is a contradiction. So  the proof is finished. 
\end{proof}
\begin{proposition}
	\label{phiuzero}
	If $\vec{0}\neq \phi$ and $\phi(u)=0$ then $\text{MI}^\phi(S)=\text{MI}(S)-1$.
\end{proposition}
\begin{proof}
	First, we observe that $\text{MI}(S)\geq 1$. Suppose the claim is false. Then $H_{-}=\emptyset$ from Theorem \ref{funddecom}. Then, $u\in H_0$ and 
	\[\phi=\mathcal{S}(u)=\vec{0},\] 
	which is contradiction. So the claim is true. 
	
	By Lemma \ref{maximalnotinkernel}, there is a maximal subspace $W$ on which $S(\cdot, \cdot)$ is negative definite and $W\not\subset \text{Ker}(\phi)$. Let $W_1=\text{Ker}(\phi\mid_{W})$ then $W_1$ has codimension one. Thus, there is a nonzero vector $u_0\in W$ such that it is $S$-perpendicular to $W_1$, $\phi(u_0)\neq 0$, and
	\begin{align*}
	W &=\text{span}(u_0)\oplus W_1,\\
	\text{dim}(W_1) &=\text{dim}(W)-1\\
	&=\text{MI}(S)-1.
	\end{align*}
	Consequently, 
	\[H=\text{span}(u_0)\oplus W_1\oplus W^\perp.\]
	
	\textbf{Claim:} $W_1$ is a maximal subspace inside $\text{Ker}(\phi)$ such that $S(\cdot,\cdot)$ is negative definite on $W_1$. 
	
	\textbf{Proof of the claim.} The claim is proved by contradiction. If it were false then there would be $f\in \text{Ker}(\phi)$, such that $f$ is $S$-perpendicular to $W_1$, and $S(f,f)<0$. Since $f$ is $S$-perpendicular to $W_1$, by the decomposition above, 
	\[ f=k_0u_0+f_1,\]
	for some constant $k_0$ and $f_1\in W^\perp$. 
	 We also observe that
	 \[S(u,u_0)=(\mathcal{S}u)u_0=\phi(u_0)\neq 0. \] 
	 Let $v=u-\frac{S(u,u_0)}{S(u_0,u_0)}u_0$ then it is readily checked that $S(v,u_0)=0$ and $v$ is $S$-perpendicular to $W_1$. Therefore, $v\in W^\perp$ and, consequently,
	 \[f=ku+f_2,\]
	 for some constant $k$ and $f_2\in W^\perp$. We calculate
	 \begin{align*}
	 S(f,f) &= S(ku+f_2, ku+f_2)\\
	 &= k^2S(u,u)+2kS(u,f_2)+S(f_2,f_2)\\
	 &=k^2\phi(u)+2k\phi(f_2)+S(f_2,f_2).
	 \end{align*}
	 Since we have $u, f\in \text{Ker}(\phi)$ so $f_2\in \text{Ker}(\phi)$. Thus, for $f\in W^\perp$,
	 \[S(f,f)=S(f_2, f_2)\geq 0.\]
	 That is a contradiction and the claim is proved. The result then follows.  
\end{proof}

Now the S-criticality is characterized by Theorem \ref{Scriticality}. Theorems \ref{abstractMorse} and \ref{severalfunc} will follow immediately.  

\begin{proof} [Proof of Theorem \ref{abstractMorse}.] It follows from Theorem \ref{Scriticality} and Lemma \ref{atmost1}.
	
\end{proof}
Towards Theorem \ref{severalfunc}, let $\phi_1,...\phi_n$ be non-trivial continuous linear functionals. Via the Riesz representation theorem, each $\phi_i$ corresponds to $\bar\phi_i\in H$ such that, for every $u\in H$, 
\[\phi_i(u)=(\bar\phi_i, u).\]

\begin{lemma}
	$\phi_1,...\phi_n$ are linearly independent if and only if, for each $i$, $\phi_i$ is a non-trivial linear functional on $\cap_{j=1, j\neq i}^n\text{Ker}\phi_j$. 
\end{lemma}
\begin{proof}
	In this proof, all perpendicularity is with respect to the Hilbert space inner product. By the closeness of a subspace of finite codimension in a Hilbert space,
	\[\big(\cap_{j=1, j\neq i}^n\text{Ker}\phi_j\big)^\perp= \text{span}(\bar\phi_j, j\neq i). \]
	Thus, $\phi_i$ is a trivial linear functional on $\cap_{j=1, j\neq i}^n\text{Ker}\phi_j$ if and only if $\bar\phi_i\in \text{span}(\bar\phi_j, j\neq i)$. 
	The proof then follows. 
	
\end{proof}

\begin{proof}[Proof of Theorem \ref{severalfunc}.]
	 Since  $\phi_1,\phi_2, \cdots, \phi_n$ are linearly independent, each is non-trivial on $H$. Furthermore, $u_1, u_2,...u_n$ are also linear independent. As the matrix $S(u_i,u_j)$ is symmetric, it can be diagonalized by a basis $u_1',...u_n'$ so that for $i\neq j$
	\[ S(u_i',u_j')=0.\]
	Note that $u_1, u_2,...u_n$ and $u_1',...u_n'$ span the same subspace. By linear algebra, for $\mathcal{S}u_i'=\phi_i'$, 
	\[\cap_{i=1}^n\text{Ker}\phi_i=\cap_{i=1}^n\text{Ker}\phi_i'.\]
	Then
	\[\text{MI}^{\phi_1,...,\phi_n}(S)=\text{MI}^{\phi_1',...,\phi_n'}(S)=(...(\text{MI}^{\phi_1'})...)^{\phi_n'}(S)\]
	Since $u_i'$ and $u_j'$ are $S$-perpendicular for $i\neq j$, $\phi_i'(u_j')=S(u_i', u_j')=0$. Therefore, 
	\[u_j'\in \cap_{i=1, i\neq j}^n\text{Ker}\phi_i'.\]
 Since  $\phi_1,\phi_2, \cdots, \phi_n$ are linearly independent, so are $\phi_1',\phi_2', \cdots, \phi_n'$. 
 
 	We next proceed by induction. The statement is true for $n=1$ by Theorem \ref{abstractMorse}. Suppose then it is true for $n=k$ and we consider the case $n=k+1$. Let \[H_k=\cap_{i=1}^k\text{Ker}\phi_i'\]
 then $H_k$ is subspace of finite co-dimension in a Hilbert space. Thus, $H_k$ is a Hilbert space. By observations above, $u_{k+1}'\in H_k$ and $\mathcal{S}u_{k+1}'=\phi_{k+1}'$. Also, as $\phi_1',\phi_2', \cdots, \phi_n'$ are linearly independent, $\phi_{k+1}'$ is a non-trivial continuous linear functional on $H_k$.  Applying Theorem \ref{abstractMorse} finishes the proof. 
 		
\end{proof}
\subsection{Nullity}
The purpose of this section is to prove Theorems \ref{abstractnullity} and \ref{severalfuncnullity}. Before giving their proof, we will collect useful results. Recall that the nullity is the dimension of the kernel of $\mathcal{S}$. Each element in the kernel is $S$-perpendicular to every element in the Hilbert space. So it makes sense, in this section, to consider the perpendicularity, projection, and direct sum with respect to the inner product. 

For a non-trivial continuous linear functional $\phi$, let 
\begin{align*}
H_1 &:= \text{Ker}\phi,\\
N &:= H_0=\text{Ker}(\mathcal{S}),\\
N_1 &:= (H_1)_0=\text{Ker}(\mathcal{S}\mid_{H_1}).
\end{align*} 
Via the Riesz representation theorem, each $\phi$ uniquely corresponds to $\bar{\phi}\in H$ such that, for all $v\in H$,
\[\phi(v)=(\bar{\phi}, v).\] 
\begin{lemma}
	\label{lemma1null}
	$v\notin N$ and $v\in N_1$ if and only if $\phi(v)=0$, $\mathbf{S}(v)=k\bar\phi$ for some non-zero constant $k$.
\end{lemma}
\begin{proof}
	If $\mathbf{S}(v)=k\bar\phi$, $0\neq k$, then immediately $v\notin N$. Since $\phi(v)=0$, $v\in H_1$. For every $u\in H_1$,
	\[S(v, u)=\mathcal{S}v(u)=(\mathbf{S}v, u)=k(\bar\phi, u)=0.\] 
	 Thus, $v\in N_1$. 
	 
	 For the other direction, let $v\in N_1$, $v\notin N$. Thus, $v\in H_1$ and $\phi(v)=0$. Consider, for every $u\in H_1$,
	\[0=S(v,u)=(\mathbf{S}v, u).\]
	Therefore, the projection (by the inner product) of $\mathbf{S}v$ on $H_1$ is $\vec{0}$. Since $v\notin N$, $\mathbf{S}v\neq \vec{0}$. As $H$ is a direct sum (by the inner product)of $H_1\oplus \text{span}(\bar\phi)$, the result follows.
\end{proof}

\begin{proof}[Proof of Theorem \ref{abstractnullity}] There are two cases. 
	
		{\bf Case (i).} If $\phi\in \overline{\text{ran}(\mathcal{S})}$, then $\bar\phi$ is $S$-perpendicular to $\text{Ker}(\mathcal{S})$ and also $H$-inner product perpendicular to $\text{Ker}(\mathcal{S})$. Therefore $\text{Ker}(\mathcal{S})\subset \text{Ker}(\phi)$, $N\subset N_1$, and $n^\phi(S)\geq n(S)$.
		
		By Lemma \ref{lemma1null}, the inequality is proper if and only if 
		there is $u\in H$ such that 
		\begin{align*}
		\mathcal{S}u &=\phi,\\
		\phi(u) &=0.
		\end{align*}
		$u$ is unique up to an addition of an element from $N$. As a consequence, in this case, \[n^\phi(S)= n(S)+1.\]
Otherwise, 
 \[n^\phi(S)= n(S).\] 
		   
		{\bf Case (ii).} If $\phi\not\in \overline{\text{ran}(\mathcal{S})}$. Then $\bar\phi=u+s$ for $\vec{0}\neq u\in N$ and $s\in  \overline{\text{ran}(\mathcal{S})}$. By Lemma \ref{lemma1null}, $N_1\subset N$. Furthermore, 
			\[\phi(u)=(u, u+s)=(u,u)>0.\]
		Thus, $u\not\in \text{Ker}\phi$ and, consequently, $u\not\in N_1$. So the inclusion $N_1\subset N$ is proper and by Lemma \ref{atmost1}, 
		 \[n^\phi(S)= n(S)-1.\] 
		
\end{proof}

\begin{proof}[Proof of Theorem \ref{severalfuncnullity}]
	We argue as in the proof of Theorem \ref{severalfunc} using Theorem \ref{abstractnullity} as the base case. 
\end{proof}

\section{An Index Formula for Manifolds with Boundaries}
In this section, using our abstract formulation, we prove a general index formula for manifolds with boundaries, recovering results from \cite{tran16index}. Let $\Sigma$ be a smooth, orientable Riemannian manifold with boundaries. Let $\nabla, \Delta$ be the covariant derivative and Laplace operator on $\Sigma$ respectively. Let $d\mu$ and $ds$ be the induced volume form on $\Sigma$ and $\partial\Sigma$ respectively. 

Generally, on such a manifold, a second variation of some functional is associated with the following structurally general bilinear form:  
\begin{equation*}
Q(u,v) = \int_{\Sigma}\Big(\left\langle{\nabla u,\nabla v}\right\rangle-p uv\Big) d\mu-\int_{\partial \Sigma}q uv ds, 
\end{equation*} 
for smooth functions $p, q$ determined by the geometry of $\Sigma$. Let
\[J:=\Delta+p\]
be the so-called Jacobi operator. Then, via integration by parts,
\[Q(u,v) = \int_{\Sigma}\Big(-(Ju)v\Big) d\mu+\int_{\partial \Sigma}(\nabla_\eta u- qu)v ds. \] 
Here $\eta$ is an out-ward conormal vector along $\partial\Sigma$. The index of this form, $\text{MI}(Q)$, is precisely the number of negative eigenvalues of a Robin boundary consideration:
\begin{equation}
\label{Robin}
\begin{cases}
{J}u &=-\lambda u \text{ on } \Sigma,\\
\nabla_\eta u &= q u \text{ on } \partial \Sigma.
\end{cases}
\end{equation}

In \cite{tran16index}, the first author shows that $\text{MI}(Q)$ can be precisely determined by data of simpler problems. First, one consider only variations fixing the boundary. The associated fixed boundary problem is given by the following Dirichlet consideration:  
\begin{equation}
\label{FBP}
\begin{cases}
{J}v &=-\delta v \text{ on } \Sigma,\\
v &= 0 \text{ on } \partial \Sigma.
\end{cases}
\end{equation}
The Dirichlet eigenvalues can be characterized variationally. Let $H^1_0(\Sigma)$ be the Sobolev with one derivative, $L^2$-norm, and zero trace (intuitively, zero on the boundary). Let $V_k\subset H^1_0(\Sigma)$ denote a $k$-dimensional subspace. Then
\begin{align*}\label{Dirichlet}
\delta_k &=\min_{V_k}\max_{\vec{0}\neq v\in V_k} \frac{\int_\Sigma |\nabla v|^2- p v^2}{\int_{\Sigma} v^2},
\end{align*}

The influence of the boundary is, then, captured by the Jacobi-Steklov problem. See \cite{AEKS14, AM12} for the PDE theoretical foundation of this consideration. The Steklov eigenvalue problem, associated with the Laplace operator instead of $J$, for free boundary minimal surfaces also received tremendous interests recently, for example, \cite{FS11, FS16, petrides19}.    

Suppose that $q\in C^\infty(\partial \Sigma)$ be a non-zero non-negative function. We consider:  
\begin{equation}
\label{Jacobi-Steklov}
\begin{cases}
{J}h &=0 \text{ on } \Sigma,\\
\nabla_\eta h &= \mu q h \text{ on } \partial \Sigma.
\end{cases}
\end{equation}
The $J$-Steklov eigenvalues can be characterized variationally. Let $V_k\subset H^1(\Sigma)\cap \text{Ker}(J)$ denote a $k$-dimensional subspace, then
\begin{align*}\label{QandS}
\mu_k &=\min_{V_k}\max_{0\neq h\in V_k} \frac{\int_\Sigma |\nabla h|^2- p h^2}{\int_{\partial\Sigma} q h^2},
\end{align*}

\begin{remark}
	$\text{Ker}(\mathbf{Q})$ is essentially the eigenspace of eigenvalue $0$ of (\ref{Robin}) and is also the eigenspace of eigenvalue $1$ of (\ref{Jacobi-Steklov}).
\end{remark}
The following result, a slight generalization of \cite{tran16index}[Theorem 3.3], basically relates $\text{MI}(Q)$, the number of negative eigenvalues of (\ref{Robin}) with spectral data of (\ref{FBP}) and (\ref{Jacobi-Steklov}). 

\begin{theorem}\label{indexdecompose}
	Let $(\Sigma, \partial \Sigma)$ be a smooth compact Riemannian manifold with boundary and $q\geq 0$, $q\not\equiv 0$. Then $\text{MI}(Q)$ is equal to  
	\[a+b.\]
	Here $a$ is the number of non-positive eigenvalues of (\ref{FBP}) counting multiplicity; $b$ is the number of eigenvalues smaller than $1$ of (\ref{Jacobi-Steklov}) counting multiplicity.  
\end{theorem}
We give a proof based on the theory we just developed.
\begin{proof}
	By Lemma \ref{vtoH}, it suffices to consider the index form $Q(\cdot, \cdot)$ on the Hilbert space $H^1(\Sigma)$. Let $\mathcal{Q}$ be the associated operator from that Hilbert space to its continuous dual. 
	
	Let $u_1,...u_a$ be a maximal set of independent orthonormal eigenfunctions of (\ref{FBP}) with non-positive eigenvalues. Let $h_1,...,h_b$ be a maximal set of independent orthonormal eigenfunctions of (\ref{Jacobi-Steklov}) with eigenvalues smaller than $1$. Also, without loss of generality, we can assume that  $u_1,...u_a$ are $L^2(d\mu)$ mutually perpendicular to each others and $h_1,...,h_b$ are mutually $L^2(q ds)$ mutually perpendicular to each others.  
	Let 
	\begin{align*}
	\phi_i &=\mathcal{Q}u_i,\\
	\varphi_j &=\mathcal{Q}h_j.
	\end{align*}
	
	One readily checks that $\phi_1,...,\phi_a$, $\varphi_1,... \varphi_b$ are linearly independent and mutually $Q$-perpendicular to each other. Then, it follows from Theorem \ref{severalfunc} that
	$$\text{MI}^{\phi_1,...\phi_a,\varphi_1,...\varphi_b}(Q)=\text{MI}(Q)-a-b.$$
	So the rest of the proof is to prove  $\text{MI}^{\phi_1,...\phi_a,\varphi_1,...\varphi_b}(Q)=0$. 
	
	Indeed  let $v\in \cap_{i=1}^a\text{Ker}(\phi_i)  \cap\cap_{j=1}^b\text{Ker}(\varphi_j)$. We have
	\begin{align*}
	0 &=\phi_i(v)\\
	&=\mathcal{Q}(u_i)v\\
	&=Q(u_i, v)\\
	&=\int_{\Sigma}\lambda_i u_i v d\mu+\int_{\partial\Sigma}(\nabla_\eta u_i) v ds.
	\end{align*}
	
	\textbf{Claim:} There is a function $h\in H^1(\Sigma)$ such that 
	\begin{align*}
	Jh &=0 \text{ on } \Sigma,\\
	h &= v \text{ on } \partial \Sigma. 
	\end{align*}
	
	\textbf{Proof of the claim.} If $0$ is not an eigenvalue of (\ref{FBP}), then the associated homogeneous system to one above has no solution. Then the claim follows from the Fredholm alternative \cite{evans10}. 
	
	If $0$ is an eigenvalue of (\ref{FBP}), then by the calculation above, $\lambda_i=0$, 
	\[\int_{\partial\Sigma}(\nabla_\eta u_i) v ds=0.\]
	By a variation of the Fredholm alternative \cite[Lemma 2.5]{tran16index}, the claim also follows. 
 	 
	Thus, the claim holds and let $u=v-h$. Then 
	\begin{align*}
	u &=0 \text{ on } \partial \Sigma. 
	\end{align*}
	Furthermore,
	\begin{align*}
	\phi_i(u) &=\phi_i(v)-\phi_i(h)\\
	&=-\phi_i(h)\\
	&=-Q(u_i,h)\\
	&= \int_{\Sigma} u_i Jh d\mu-\int_{\partial\Sigma}(\nabla_\eta h-qh) u_i ds\\
	&=0.
	\end{align*}
	Thus, $u$ is $L^2(d\mu)$ perpendicular to each $u_i$ $i=1,...a$. 
	
	Similarly, 
	\begin{align*}
	\varphi_j(h) &=\varphi_j(v)-\varphi_i(u)\\
	&=-\varphi_i(u)\\
	&=-Q(u,h_j)\\
	&= \int_{\Sigma} u Jh_j d\mu-\int_{\partial\Sigma}(\nabla_\eta h_j-qh_j) u ds\\
	&=0.
	\end{align*}
	On the other hand, 
	\begin{align*}
	\varphi_j(h) &=Q(h_j, h)\\
	&=\int_{\partial\Sigma}(\nabla_\eta h_j-qh_j) h \\
	&= (\mu_j-1)\int_{\partial\Sigma}q h_j h.
	\end{align*}
	Thus, $u$ is $L^2(qds)$ perpendicular to each $h_j$ $j=1,...b$. By the variational characterization of (\ref{FBP} and (\ref{Jacobi-Steklov}) and $Q(u,h)=0$, 
	\[Q(u+h,u+h)=Q(u,u)+Q(h,h)\geq 0. \]
	Therefore,  $\text{MI}^{\phi_1,...\phi_a,\varphi_1,...\varphi_b}(S)=0$ and the proof is finished. 
	
\end{proof}

\bibliographystyle{plain}
\bibliography{bioMorse}

\end{document}